\documentclass[11pt,reqno]{amsart}
\usepackage{amssymb, amscd, amsmath, enumerate,verbatim}
\usepackage{amscd}
\usepackage{color}
\usepackage{graphicx}
\usepackage[utf8]{inputenc}

\input xy
\xyoption{all}

\numberwithin{equation}{subsection}
\newtheorem{theorem}{Theorem}[subsection]
\newtheorem{lemma}[theorem]{Lemma}
\newtheorem{proposition}[theorem]{Proposition}
\newtheorem{corollary}[theorem]{Corollary}

\theoremstyle{definition}

\newtheorem{definition}[theorem]{Definition}
\newtheorem{remark}[theorem]{Remark}

\newtheorem{example}[theorem]{Example}
\newtheorem{examples}[theorem]{Examples}
\newtheorem{hypotheses}[theorem]{Hypotheses}

\newcommand{\mk}{\mathbf{k}}

\newcommand{\A}{\mathcal{A}}

\newcommand{\C}{\mathcal{C}}
\newcommand{\D}{\mathcal{D}}
\newcommand{\E}{\mathcal{E}}
\newcommand{\F}{\mathcal{F}}
\newcommand{\G}{\mathcal{G}}

\newcommand{\Lcal}{\mathcal{L}}
\newcommand{\M}{\mathcal{M}}

\newcommand{\Ocal}{\mathcal{O}}
\newcommand{\Pcal}{\mathcal{P}}

\newcommand{\Scal}{\mathcal{S}}
\newcommand{\T}{\mathcal{T}}

\newcommand{\W}{\mathcal{W}}
\newcommand{\X}{\mathcal{X}}
\newcommand{\Y}{\mathcal{Y}}

\newcommand{\Simpl}{\mathbf{\Delta}}

\newcommand{\Sset}{s\mathbf{S}}

\newcommand{\Alg}{\mathbf{Alg}}

\newcommand{\Mon}[1]{\mathbf{Mon}(#1)}

\newcommand{\sMod}{\mathbf{\Sigma Mod}}

\newcommand{\Op}{\mathbf{Op}}

\newcommand{\Sp}{\mathbf{Sp}}

\newcommand{\Top}{\mathbf{Top}_*}

\newcommand{\loc}[2]{\mathcal{#1}[\mathcal{#2}^{-1}]}

\newcommand{\rdf}{\mathbb{R}}

\newcommand{\quis}{\textrm{quasi-isomorphism}}

\newcommand{\Cochains}[1]{\mathbf{C}^*(#1)}
\newcommand{\Cochainsp}[1]{\mathbf{C}^{\geq b}(#1)}
\newcommand{\Cochainsk}{\mathbf{C}^{\geq 0}(\mk)}

\newcommand{\Sheaf}{\mathbf{Sh}}
\newcommand{\PrSheaf}{\mathbf{PrSh}}
\newcommand{\Sh}[2]{\mathbf{Sh}(#1,#2)}
\newcommand{\PrSh}[2]{\mathbf{PrSh}(#1,#2)}

\newcommand{\y}{\mathbf{y}}

\newcommand{\dirlim}{\underset{\longrightarrow}{\mathrm{colim}}\,}

\newcommand{\holim}{\underset{\longleftarrow}{\mathrm{holim}}\,}

\newcommand{\unidad}{\mathbf{1}}

\newcommand{\Homsheaf}{{\mathcal Hom}\,}

\newcommand{\op}{\mathrm{op}}

\newcommand{\id}{\mathrm{id}}

\newcommand{\TW}{\mathrm{TW}}

\newcommand{\simple}{\mathbf{s}}

\newcommand{\Lie}{\mathcal Lie}
\newcommand{\Com}{\mathcal Com}

\newcommand{\Endop}{\mathrm{End}}

\newcommand{\mc}[1]{\mathcal{#1}}
\newcommand{\mrm}[1]{\mathrm{#1}}
\newcommand{\mbf}[1]{\mathbf{#1}}
\newcommand{\mbb}[1]{\mathbb{#1}}
\newcommand{\Dl}{\ensuremath{\Simpl} }

\newcommand{\ds}{\displaystyle}

\begin{document}
\begin{large}

\title[Godement resolution]{Godement resolution and operad sheaf homotopy theory}

\author{Beatriz Rodr\'{\i}guez Gonz\'{a}lez}
\address{ICMAT\\ CSIC-Complutense-UAM-CarlosIII\\ Campus Cantoblanco, UAM. 28049 Madrid, Spain.}

\author{Agust\'{\i} Roig}
\address{Dept. Matem\`{a}tica Aplicada I\\ Universitat Polit\`{e}cnica de Catalunya, UPC, and BGSMath \\ Diagonal 647, 08028 Barcelona, Spain.}

\thanks{First named author partially supported by ERC Starting Grant project TGASS and by contracts SGR-119 and FQM-218. Second named author partially supported by projects MTM2012-38122-C03-01/FEDER and 2014 SGR 634}

\begin{abstract} We show how to induce products in sheaf cohomology for a wide variety of coefficients: sheaves of dg commutative and Lie algebras, symmetric $\Omega$-spectra, filtered dg algebras, operads and operad algebras.
\end{abstract}

\date{\today}
\maketitle
\tableofcontents

\section{Introduction}

\subsection{Products in sheaf cohomology} How to induce products in sheaf cohomology and other derived functors of sheaves has been a matter of interest since the beginning of sheaf theory \cite{Go}, and keeps receiving attention persistently \cite{GK}, \cite{Ja1}, \cite{Sw}, etc.

With classic abelian sheaves, the problem appears quite blatantly, because the tensor product of two injective resolutions is not necessarily an injective resolution, and so products don't go through them. Despite this fact, people managed to produce products in sheaf cohomology when needed in Hodge theory \cite{N}, or deformation quantization \cite{Be}, for instance. Now that operad theory and its applications have reached the state of sheafification \cite{Hin}, \cite{KM}, as also is the case with symmetric spectra \cite{Ja2}, products in sheaf cohomology, and other derived functors, might deserve some attention again.

The problem we address in this paper is the following: let $\X$ be a Grothendieck site, $\D$ a monoidal category, $\Op^\D$ the category of operads in $\D$ and $\Sheaf (\X, \D)$ the category of sheaves on $\X$ with coefficients in $\D$. Given a sheaf of operads $\Pcal \in \Sheaf (\X, \Op^\D)$ and a sheaf of $\Pcal$-algebras $\A$, how can we transfer the products, i.e., structural morphisms, from $\Pcal$ and $\A$ to their cohomologies $\rdf \Gamma (X, \Pcal)$ and $\rdf \Gamma (X, \A)$?

In \cite{RR} we gave necessary and sufficient conditions on the Gro\-then\-dieck site $\X$ and the category of coefficients $\D$ in order that the Godement cosimplicial resolution produces a fibrant model $\mbb{H}_{\X} (\F)$ for every sheaf $\F $ and hence it can be used to transfer homotopical structure from $\D$ to the category of sheaves $\Sheaf (\X , \D)$. In particular, sheaf cohomology with coefficients in $\F$, $\rdf \Gamma (X, \F)$ can be computed as $\Gamma (X, \mbb{H}_{\X} (\F))$.

Building on our main result in \cite{RR}, we provide here a positive answer to the question about products: if a sheaf $\F$ comes equipped with some kind of \lq\lq product" $\F \otimes \F \longrightarrow \F$, this product passes to cohomology $\rdf \Gamma (X, \F) \otimes \rdf \Gamma (X, \F) \longrightarrow \rdf \Gamma (X, \F)$ through the Godement cosimplicial resolution, whenever this resolution provides a fibrant model for  $\F$. As a consequence, for instance, for a sheaf of operads $\Pcal$ and a sheaf of $\Pcal$-algebras $\A$, $\rdf \Gamma (X, \Pcal)$ is an operad and $\rdf \Gamma (X, \A)$ is a $\rdf \Gamma (X, \Pcal)$-algebra.

\subsection{The Godement resolution is a monoidal functor} This is true for a simple reason: for every monoidal coefficient category $\D$ for which it can be defined, the Godement cosimplicial resolution is a \textit{monoidal} functor, under mild assumptions.

This seems to have been well known for a while, at least for some particular cases of $\D$ (\cite{N}, \cite{Le}), even if the word \lq\lq monoidal" was not yet available to name the phenomenon (\cite{Go}, \cite{Sw}); but, to the best of our knowledge, it has not explicitly been stated for general coefficients $\D$. That is, in a nutshell, what we do here: prove that, when available, the Godement cosimplicial resolution is a monoidal functor, check when it provides a fibrant model for every sheaf, how this fibrant model inherits no matter of multiplicative structure the original sheaf can bring with it, and passes it to sheaf cohomology and other derived functors.

We show, particularly, that our scheme applies to sheaves of commutative and Lie dg algebras (\cite{N}, \cite{Hin}, \cite{Be}, examples (3.1.8)), the operad sheaf of multidifferential operators modelling systems of partial differential equations (polynomial in the derivatives, \cite{KM}, example (3.2.3)) and sheaves of symmetric omega spectra (\cite{Ja2}, example (3.3.5)).

But, to encompass all these situations, we need to go slightly further than we have just briefly summarized. Since the tensor product of sheaves is not necessarily a sheaf, in order to have a monoidal structure on $\Sheaf (\X, \D)$ induced by the one on $\D$ and hence be in a position to talk about sheaves with \lq\lq products", we need an associated sheaf functor to make a sheaf $\F \otimes \F$ out of the presheaf $U \mapsto \F (U) \otimes \F (U)$. Indeed, the associated sheaf functor is available for a large class of coefficients categories $\D$. Namely, it only demands for $\D$ to have all small limits, filtered colimits and the commutation of finite limits and the later ones (see \cite{Ul}). But these simple conditions would rule out some interesting $\D$, such as, Kan complexes or symmetric omega spectra, with their exciting commutative smash product \cite{HSS}.

Therefore, we have taken some pains in order to include these cases in our work too. So, our first version of the main result (\ref{maintheorem2}) says that, with no need for an associated sheaf, if the Godement resolution provides fibrant models for $\Sheaf (\X, \D)$, so does it for the categories of sheaves of operads and operad algebras in $\D$.

Nevertheless, the category of symmetric omega spectra seems elusive and this general result applies neither: it needs a slight refinement, which we state and prove in (\ref{CEAlgP}). The reader having an associated sheaf functor for her category of coefficients $\D$ can skip this discussion and go directly to the strongest version of our result (\ref{lastone}) in the final section (3.4).

\subsection{Compatible, descent and CE-categories} Our results in this paper are stated in the language of \textit{descent} and \textit{CE-categories} (\cite{Ro1}, \cite{GNPR}, \cite{P}; see also \cite{C}, \cite{CG1}, \cite{CG2}, \cite{CG3} for other examples of their use) and are based on our previous main theorem (4.14) in \cite{RR}. We now go over these concepts and recall the theorem.

\begin{theorem}\label{maintheoremRR} Let $\X$ be a Grothendieck site and $(\D, \mathrm{E})$ a descent
category satisfying the hypotheses (\ref{hipotesis1}) below. Then, the
following statements are equivalent:
\begin{enumerate}
 \item[\emph{(1)}] $(\Sh{\X}{\D}, \Scal, \mc{W})$ is a right Cartan-Eilenberg category and for every sheaf $\F \in \Sh{\X}{\D}$,
$\rho_\F : \F \longrightarrow \mbb{H}_{\X} (\F)$ is a CE-fibrant model.
 \item[\emph{(2)}] For every sheaf $\F\in \Sh{\X}{\D}$, $\rho_\F : \F\longrightarrow \mbb{H}_{\X} (\F)$ is in $\mc{W}$.
 \item[\emph{(3)}] The simple functor commutes with stalks up to equivalences.
 \item[\emph{(4)}] For every sheaf $\F\in \Sh{\X}{\D}$, $\mbb{H}_{\X} (\F)$ satisfies Thomason's descent; that is, $\rho_{\mbb{H}_{\X} (\F)} :
\mbb{H}_{\X} (\F) \longrightarrow \mathbb{H}^2_\X (\F)$ is in $\Scal$.
\end{enumerate}
\end{theorem}

\begin{hypotheses}\label{hipotesis1} The required hypotheses are the following:

\begin{enumerate}
\item[(G0)] $\X$ is a Grothendieck site with a set $X$ of enough points.
\item[(G1)] $\D$ is closed under filtered colimits and arbitrary products.
\item[(G2)] Filtered colimits and arbitrary products in $\D$ are $\mathrm{E}$-exact.
\end{enumerate}

\end{hypotheses}

Here, $\mathrm{E}$-exact means preserving equivalences (\cite{RR}, definition (4.1)).

We called a pair $(\X, \D)$ satisfying the equivalent conditions of this theorem \emph{compatible}. Let us explain the main definitions involved in our result.

\begin{enumerate}
\item A \emph{descent category} $(\D, \mathrm{E})$ consists, roughly speaking, of three pieces of data: a category $\D$, a class of distinguished morphisms $\mathrm{E}$, called \emph{equivalences}, and a \emph{simple} functor $\simple : \Simpl\D \longrightarrow \D$. The reader may safely anchor her ideas for the moment by thinking of $\D$ as being the category of cochain complexes of abelian groups, $\mathrm{E}$ the class of cohomology isomorphisms and $\simple$ as the total complex of a double complex (see \cite{RR} and \cite{Ro1} for the precise definitions, cf. \cite{GN}). Alternatively, she can think of $\simple$ as her preferred $\holim$ functor.
\item A \emph{Cartan-Eilenberg} structure on a category $(\C, \Scal, \W)$ is the minimum amount of data necessary in order to talk about \emph{fibrant} objects and derive functors. These data are two classes of distinguished morphisms $\Scal \subset \W$, \emph{strong} and \emph{weak} equivalences, and a class of fibrant objects. For sheaves, these classes are the \emph{global} equivalences (object-wise defined) and \emph{local} equivalences (fibre-wise defined), respectively (see \cite{RR} and \cite{GNPR} for the precise definitions).
\item The fibrant objects turn out to be the  \emph{hypercohomology} sheaves $\mathbb{H}_\X(\F) = \simple G^\bullet (\F)$: resulting from the application of the simple functor $\simple$ to the cosimplicial Godement resolution $G^\bullet \F$ (see \cite{RR} for the details, cf. \cite{Mit} and \cite{Th} for sheaves of spectra).
\end{enumerate}

So, our theorem proves the following to be equivalent: (a) the fact that the hypercohomology sheaf $\mathbb{H}_\X(\F)$ is a fibrant resolution of $\F$, (b) the existence of a Cartan-Eilenberg structure in the category of sheaves $\Sheaf (\X, \D)$, (c) Thomason's descent and (d) a condition expressed in terms of the \lq\lq raw" data: the simple functor $\simple$, the equivalences $\mathrm{E}$ of the coefficient category $\D$, and the topology of the site, i.e., the fibres of the sheaves.

This last condition can also be read as follows: the universal map $(\simple G^\bullet\F )_x \longrightarrow \simple (G^\bullet\F_x)$ must be an equivalence in $\D$. For instance, for the category of bounded below cochains complexes of sheaves, this condition is trivially satisfied since the total complex commutes with colimits. Another way to look at this condition: since all simple functors are realizations of homotopy limits \cite{Ro2}\footnote{So, in particular, example (2.1.6)  proves that the \emph{Thom-Whitney simple functor} of \cite{HS} and \cite{N} is a realization of a homotopy limit too.}, it requires for some filtered colimits indexed by $\X$ to commute up to equivalence with homotopy limits in $\D$: $\dirlim \holim  \simeq \holim \dirlim $. So, in order the Godement resolution produces a fibrant model for every sheaf the topology of the site and the homotopy structure of the coefficient category must be intertwined (cf \cite{MV}, proposition (1.61), for sheaves of simplicial sets).

Some examples of compatible pairs we showed in (\cite{RR}):
\begin{enumerate}
\item The category of uniformly bounded below cochain complexes $\Cochainsp{\A}$ of an $(AB4)^*$ and $(AB5)$ abelian category $\A$, taking the \quis\ as equivalences $\mathrm{E}$ and the total complex of a double complex as the simple functor, is a descent category compatible with any site.
\item The category of unbounded cochains complexes $\Cochains{R}$ of $R$-modules, with the same equivalences and with the total (product) complex as simple functor, is a descent category compatible with any finite cohomological dimension site.
\item The category of filtered complexes $\textbf{F}\Cochainsp{\A} $, with $\mathrm{E}_r$-\quis\ as equivalences and the total complex as simple functor, is compatible with any site.
\item The category of Kan simplicial complexes $\Sset_f$, with $\mathrm{E}$ as the weak homotopy equivalences and $\holim$ as simple functor, is a descent category compatible with any site of finite type.
\item The category of fibrant spectra $\Sp_f$, with $\mathrm{E}$ the stable weak equivalences and its $\holim$ as simple functor, is a descent category compatible with any site of finite type.
\end{enumerate}

In this paper we apply these techniques and extend our main theorem (\ref{maintheoremRR}) and its corollaries for compatible pairs $(\X,\D)$ to sheaves of operads and operad algebras in $\D$. Precisely, if $\E$ is a category of operads or operad algebras in $\D$, our main results says -with different flavors; theorems (3.1.4), (3.2.5), (3.3.8) and (3.4.9)- that also $(\X, \E)$ is a compatible pair and therefore:

\begin{itemize}\label{PropertiesCompat}
\item[$\bullet$] For every sheaf $\F$, the natural arrow $\rho_\F : \F \longrightarrow \mathbb{H}_\X(\F)$ is a fibrant model of $\F$. Or, equivalently, $(\Sh{\X}{\E},\Scal,\W)$ is a CE-category with resolvent functor $(\mbb{H}_{\X},\rho)$.
\item[$\bullet$] The localized category $\loc{\Sh{\X}{\E}}{\W}$ is naturally equivalent to $\loc{\Sh{\X}{\E}_{\mrm{fib}}}{\Scal}$.
\item[$\bullet$] The CE-fibrant objects of $\Sh{\X}{\E}$ are precisely those sheaves satisfying Thomason's descent.
\item[$\bullet$] Derived sections $\mbb{R}\Gamma(U,-)$ and derived direct image functor $\mbb{R}f_*$ may be computed by precomposing with $\mbb{H}_{\X}$.
\item[$\bullet$] The hypercohomology sheaf $\mbb{H}_{\X}$ is a `homotopical' sheafification functor that gives an equivalence $\loc{\Sh{\X}{\E}}{\W}\simeq \loc{\PrSh{\X}{\E}}{\W}$.
\end{itemize}

For that, our strategy is quite simple: in section 2, we show how categories of operads and operad algebras inherit descent structures from a given \emph{monoidal} descent categoy $\D$. Next, in section 3, we prove that, thanks to our main theorem (\ref{maintheoremRR}), the Godement resolution provides fibrant models for sheaves of operads and operad algebras whenever it does for sheaves with coefficients in $\D$.

\section{Descent categories with products}

\subsection{Descent and monoidal structures}

\subsubsection{} Let  $(\C , \otimes,\unidad)$ be a symmetric monoidal
category. A (lax) {\it symmetric monoidal functor\/} between symmetric monoidal categories $(\C, \otimes ,\unidad ), (\D , \otimes , \unidad) $ is a functor $F: \C \longrightarrow \D $ together with a natural
transformation, \emph{Künneth morphism}, $ \kappa_{XY} : FX
\otimes FY \longrightarrow F(X\otimes Y)$ and a \emph{unit} morphism
$\eta : \unidad \longrightarrow F\unidad$ in $\D$, compatible with the
associativity, commutativity and unit constraints.

\begin{definition}\label{MonoidalDescentDefi} A {\it monoidal descent category\/} is a descent category $(\D, \mrm{E}, \simple , \mu , \lambda )$ (see \cite{Ro1}, \cite{RR}) such that:

\begin{enumerate}
\item[(M1)] $\D$ is a symmetric monoidal category.
\item[(M2)] The simple functor $\simple : \Simpl\D \longrightarrow \D$ is a symmetric
monoidal functor.
\item[(M3)] The natural transformations $\mu: \simple\simple\longrightarrow\simple\mathrm{D}$ and  $\lambda:\id_{\D}\longrightarrow \simple c$ are monoidal.
\end{enumerate}
\end{definition}

\subsubsection{}\label{examples1} Let us describe some examples of monoidal descent categories.

\begin{example}\label{examplemodelcategories1} \textbf{Simplicial symmetric monoidal model categories.} Let $\M$ be a a simplicial symmetric monoidal model category in which all objects are fibrant. For instance, $\M$ can be the category of compactly generated pointed topological spaces, $\Top$, with the smash product.

\begin{proposition}\label{topologicalSpaces} Every simplicial symmetric monoidal model category in which all objects are fibrant is a monoidal descent category.
\end{proposition}

\begin{proof} By \cite[3.2]{Ro1}, $\mc{M}_f = \M$ is a descent category where $\mrm{E}$ is the class of weak equivalences of
$\mc{M}$ and the simple functor is the Bousfield-Kan \cite{BK} homotopy limit, $\holim :\Dl\mc{M}_f\longrightarrow\mc{M}_f$
(see also \cite{Hir}). Given a cosimplicial object $X$,
$$\holim X=\int_n (X^n)^{\mrm{N}(\Dl\downarrow n)} \ ,$$
where $\mrm{N}(\Dl\downarrow n)$ is the nerve of the over-category $(\Dl\downarrow n)$
and $(X^m)^{\mrm{N}(\Dl\downarrow n)}$ is constructed using the simplicial structure on $\mc{M}$.

Recall also that $\mu$ and $\lambda$ are easily defined using the fact that a functor
$F:\mc{B}\longrightarrow\mc{C}$ induces a natural map
$
\holim_{\mc{C}}X\longrightarrow\holim_{\mc{B}}F^\ast X  = \int_b X(F(b))^{\mrm{N}(\mc{B}\downarrow b)}
$
which is defined by the maps $X(F(b))^{\mrm{N}(\mc{C}\downarrow F(b))}\longrightarrow X(F(b))^{\mrm{N}(\mc{B}\downarrow b)}$, induced by
${F}:(\mc{B}\downarrow b)\longrightarrow (\mc{C}\downarrow F(b))$. Then:

\begin{itemize}
\item[$\bullet$ ] $\mu$ is obtained from the diagonal $d:\Dl\longrightarrow \Dl\times\Dl$, that induces for each
$Z\in\Dl\Dl\mc{M}_f$
$$\holim_{\Dl}\holim_{\Dl} Z \simeq \holim_{\Dl\times\Dl}Z\longrightarrow\holim_{\Dl} d^\ast Z = \holim_{\Dl}\mrm{D}Z \ ,$$
where the first isomorphism follows from the Fubini property of $\holim$ (see \cite[XI.4.3]{BK}).
\item[$\bullet$ ] $\lambda$ is obtained from $l:\Dl\longrightarrow \ast$. It induces $\lambda_A:A\simeq\holim_{\ast}A\longrightarrow \holim_{\Dl}cA$ for each
$A\in\mc{M}_f$.
\end{itemize}

As for the K\"unneth morphism for  $\holim$,  let $X$, $Y$ be cosimplicial objects in $\M$.
Since all ends are (lax) monoidal functors, in particular there is a canonical morphism
$e:\left(\int_n (X^n)^{\mrm{N}(\Dl\downarrow n)}\right)\otimes \left(\int_n (Y^n)^{\mrm{N}(\Dl\downarrow n)}\right)
\longrightarrow \int_n (X^n)^{\mrm{N}(\Dl\downarrow n)}\otimes (Y^n)^{\mrm{N}(\Dl\downarrow n)}$. The diagonal embedding $\mrm{N}(\Dl\downarrow n) \longrightarrow \mrm{N}(\Dl\downarrow n)\times \mrm{N}(\Dl\downarrow n)$
produces the natural morphisms $d^n : (X^n\otimes Y^n)^{\mrm{N}(\Dl\downarrow n)\times \mrm{N}(\Dl\downarrow n)}\longrightarrow (X^n\otimes Y^n)^{\mrm{N}(\Dl\downarrow n)}$
which assemble into a natural morphism $d = \int_n d^n $ between the corresponding ends. Hence, there are natural morphisms

$$
f^n : (X^n)^{\mrm{N}(\Dl\downarrow n)}\otimes (Y^n)^{\mrm{N}(\Dl\downarrow n)}\longrightarrow (X^n\otimes Y^n)^{\mrm{N}(\Dl\downarrow n)\times \mrm{N}(\Dl\downarrow n)}
$$

coming from the compatibility between the tensor product and the simplicial action, that assemble into $f = \int_n f^n$. And the K\"{u}nneth morphism for $\holim$ is the composition $ d\circ f\circ e : \holim X \otimes \holim Y\longrightarrow \holim X\otimes Y$. The fact that $\holim$, $\mu$ and $\lambda$ are monoidal follows directly from their constructions.
\end{proof}
\end{example}

\begin{remark}
The category  $\mathbf{C}^{\geq 0}({\A})$  of non-negatively graded cochain complexes over a symmetric monoidal abelian category $\A$ is symmetric monoidal with the usual tensor product of complexes. It also has a descent structure in which the simple
functor is the total complex of the double complex obtained through the Moore functor
$M:\Dl\mathbf{C}^{\geq 0}({\A})\longrightarrow \mathbf{C}^{\geq 0}\mathbf{C}^{\geq 0}({\A})$ (see \cite[3.4]{Ro1}). However, it is \textit{not} a symmetric monoidal descent category. The reason is that the K\"{u}nneth morphism for the simple functor involves the Alexander-Whitney map and hence is not symmetric, but just up to homotopy. We know two ways to circumvent this problem:

\begin{itemize}
\item[(a)] replacing the category of cochain complexes $\mathbf{C}^{\geq 0}({\A})$  with the equivalent one of cosimplicial objects $\Dl\A$ (see Example (\ref{cosimplicialsymdescent}) below).
\item[(b)] restricting ourselves to cochain complexes of vector spaces over a field of characteristic zero (see Example (\ref{simpleTW}) below).
\end{itemize}
\end{remark}

\begin{example}\label{cosimplicialsymdescent}\textbf{Cosimplicial objects.} Let $\A$ be a symmetric monoidal abelian category and $\Dl\A$ its category of cosimplicial objects, with the degree-wise monoidal structure. $\Dl\A$ supports a natural descent structure defined as follows.

\begin{itemize}
 \item The weak equivalences are the inverse image under the Moore complex functor $M:\Dl \mc{A}\longrightarrow \mathbf{C}^{\geq 0}(\mc{A})$ of the quasi-isomorphisms (cohomology isomorphisms) in $\mathbf{C}^{\geq 0}(\mc{A})$.
 \item The simple functor is the diagonal $\mrm{D}:\Simpl\Simpl\A\longrightarrow \Simpl\A$.
 \item $\mu$ and $\lambda$ are the identity natural transformations.
\end{itemize}

All the axioms of descent category (\cite{RR}, definition (2.1)) are easy consequences of the definitions, except (S4) which follows from the dual of the Eilenberg-Zilber-Cartier theorem (\cite[2.9]{DP}). The simple functor is clearly symmetric monoidal, and $\mu$ and $\lambda$ are monoidal. Consequently, $\Dl\A$ is a monoidal descent category.  By the Dold-Kan correspondence, it is equivalent to $\mathbf{C}^{\geq 0}({\A})$ as an abelian category and also as a descent one, because the simple functors commute with the equivalence of categories.
\end{example}

\begin{example}\label{simpleTW}\textbf{Bounded complexes of vector spaces.}   Let $\mk$ be a field of characteristic zero. The category $\Cochainsk$ of non-negatively graded cochain complexes of $\mk$-vector spaces is a monoidal descent category, taking as equivalences the \quis\ and as simple functor the \textit{Thom-Whitney simple} $\simple_{\TW} : \Simpl\Cochainsk \longrightarrow
\Cochainsk$ (\cite{N}, cf \cite{HS}), whose definition we proceed to recall.

Let $\mrm{L}^*_\bullet = \left\{ L^*_p\right\}_{p\geq 0}$ be the
simplicial commutative dg algebra which in simplicial degree $p$ is the algebra of polynomial differential forms on the hyperplane $\sum_{k=0}^p x_k=0$ of the affine space $\mathbb{A}^{p+1}_\mk$:

$$
{L}^*_p = \displaystyle\frac{\Lambda(x_0,\ldots ,x_p,dx_0,\ldots,dx_p)}{\left( \sum x_i -1 ,\, \sum dx_i\right)} \ ,
$$

where $\Lambda(x_0,\ldots ,x_n,dx_0,\ldots,dx_p)$ is the free commutative dg algebra generated by $\{x_k\}_{k=0,\dots ,n}$ in degree 0 and by $\{dx_k\}_{k=0,\dots , n}$ in degree 1. By forgetting the multiplicative structure, $\mrm{L}_{\bullet}^{*} =
\left\{ L_p^*\right\}_{p\geq 0}$ can also be considered as a simplicial cochain complex. For any $V^{\bullet,*}$ in  $\Delta\Cochainsk$, the {\it Thom-Whitney simple\/} functor of $V^{\bullet,*}$ is the end of the bifunctor $\mathrm{L}_\bullet^* \otimes V^{\bullet,*} : \Simpl^{\op}_e \times
\Simpl_e \longrightarrow \Cochainsk$, $ ([p], [q]) \mapsto {L}^*_p\otimes V^{q,*}$,

$$
\simple_{\TW} (V^{\bullet,*}) = \int_p
L_p^* \otimes V^{p,*} \quad .
$$

Here $\Simpl_e$ means the \emph{strict} simplicial category. This is a monoidal functor: the K\"{u}nneth morphisms $\kappa_{VW} : \simple_{\TW}(V) \otimes \simple_{\TW}(W) \longrightarrow \simple_{\TW} (V\otimes W)$ are induced by the maps $ \kappa_p : ({L}_p^*\otimes
V^{p,*}) \otimes ({L}_p^*\otimes W^{p,*}) \longrightarrow
({L}_p^*\otimes (V^{p,*} \otimes W^{p,*} ))$ defined as $
\kappa_p ((\nu \otimes v)\otimes (\xi \otimes w)) = (-1)^{\vert
v \vert \cdot \vert \nu\vert} \nu\xi \otimes (v\otimes w)$. Symmetry follows from the commutativity of the product in ${L}_p^*$.

As for $\mu$ and $\lambda$, we have:

\begin{itemize}
 \item[$\bullet$] If $Z^{\bullet,\bullet,*}\in\Dl\Dl\Cochainsk$,
$\mu_{TW\, Z^{\bullet,\bullet,*}}:\mbf{s}_{TW}\mbf{s}_{TW}Z^{\bullet,\bullet,*}\longrightarrow
\mbf{s}_{TW}\mrm{D}Z^{\bullet,\bullet,*}=\int_p Z^{p,p,*}\otimes L^{*}_p$ is induced by the morphisms
$\xymatrix@M=4pt@H=4pt{ Z^{p,p,*}\otimes  L^{*}_p\otimes L^{*}_p \ar[r]^-{\id\otimes\tau_p}& Z^{p,p,*}\otimes L^{*}_p}$
where $\tau_p$ is the product.
 \item[$\bullet$] If
$A^*\in\Cochainsk$, the morphisms $A^* \longrightarrow A^* \otimes
 L^{*}_n$; $a\mapsto a\otimes 1$ give rise to
$\lambda_{TW\, A^*}:A^*\longrightarrow\mbf{s}_{TW}cA^*$.
\end{itemize}

To prove that this is a descent category one can consider the descent structure on $\Cochainsk$ in which the simple functor $\mbf{s}$ is induced by the total complex, and then use the fact that integration of forms yields a quasi-isomorphism of functors $\mbf{s}_{TW}\longrightarrow\mbf{s}$ (\cite{N}, \cite{HS}).
\end{example}

\begin{example}\label{ExampleFilteredComplexes}\textbf{Filtered bounded complexes of vector spaces.} Let $(V,\mrm{F})$ be a filtered positive complex of $\mk$-vector spaces. Any simplicial decreasing filtration $\varepsilon$ of $\mrm{L}_{\bullet}^*$ induces a natural filtration on $\mbf{s}_{TW}(V)$, defined on ${L}_n^{*}\otimes V^{m,*}$ by $\sum_{i+j=k}\varepsilon^i L_n^{*} \otimes \mrm{F}^jV^{m,*}$ ( \cite[(6.2)]{N}). For $r\geq 0$, we can take $\varepsilon$ to be the multiplicative filtration in which the $x_i$ have weight $0$ and the $dx_i$'s
have weight $r$. This produces the functor

$$
(\mbf{s}_{TW},\sigma_r) :\Dl \mbf{F}\Cochainsk\longrightarrow \mbf{F}\Cochainsk \ , \quad \quad  (\mbf{s}_{TW},\sigma_r)(V,\mrm{F})=(\mbf{s}_{TW}(V),\sigma_r(\mrm{F}))\ ,
$$

where
$$
\sigma_r(\mrm{F})^k(\mbf{s}_{TW}(V)^n) = \ds\bigoplus_{i+j=n}  \int_p
L^i_p \otimes \mrm{F}^{k-ri}V^{p,j} \ .
$$

Together with the natural transformations $\lambda$ and $\mu$ defined at the level of complexes as the ones for $\mbf{s}_{TW}$,
$(\mbf{s}_{TW},\sigma_r)$ endows $(\mbf{F}\Cochainsk,\mrm{E}_r)$ with a descent structure. Here $\mrm{E}_r$ is the class of
$E_r$-quasi-isomorphisms: morphisms inducing isomorphism at the $E_{r+1}$-stages of the associated spectral sequences. In particular,
$\mrm{E_0}$ is the class of graded quasi-isomorphisms.

This may be proved as follows. By \cite[2.1.12]{RR} $(\mbf{F}\Cochainsk,\mrm{E}_0)$ has a descent structure
in which the simple functor $(\mbf{s},\delta_0)$ is induced by the total complex.
The case $r=0$ is a consequence of the result \cite[(6.3)]{N}, stating that integration of forms over
the $p$-simplices yields a graded natural quasi-isomorphism $(\mbf{s}_{TW},\sigma_0)\longrightarrow (\mbf{s},\delta_0)$.
To see the general case, apply inductively the transfer lemma (2.3) \cite{RR} to the decalage filtration functor and use the fact that
$Dec \circ (\mbf{s}_{TW},\sigma_{r+1})= (\mbf{s}_{TW},\sigma_r)\circ Dec$.

Recall that $\mbf{F}\Cochainsk$ is in addition a symmetric monoidal category, with
$(\mbf{F}\otimes \mbf{G})^k (V\otimes W)^n = \oplus_{i+j=n} \sum_{s+t=k} \mbf{F}^s V^i \otimes \mbf{G}^t W^j$.
Then $(\mbf{s}_{TW},\sigma_r)$, $\lambda$ and $\mu$ are, as in the non-filtered case, symmetric monoidal. Therefore
$(\mbf{F}\Cochainsk,\mrm{E}_r)$ is a monoidal descent category.
\end{example}

\subsection{Operads}

In this section, we show how we transfer the descent structure from $\D$ to the category of operads $\Op^\D$.

\subsubsection{} Let $\Sigma$ denote the {\it symmetric groupoid\/}. The category of contravariant functors from $\Sigma$ to a category $\mathcal{C}$ is called the category of $\Sigma$-{\it modules\/} and is denoted by $\sMod^{\mathcal{C}}$. We identify its objects with sequences of objects in $\mathcal{C}$, $E=\left(E(l)\right)_{l\geq 1}$, with a right $\Sigma_l$-action on each $E(l)$.

Let $(\mathcal{C}, \otimes, \mathbf{1})$ be a symmetric monoidal category. A {\it unital $\Sigma$-operad\/} (an {\it operad\/} for short) in $\mathcal{C}$ is a $\Sigma$-module $ P$ together with a family of \textit{structure morphisms} $\gamma^P_{l;m_1,\dots ,m_l} : P(l)\otimes P(m_1) \otimes \dots \otimes P(m_l)
\longrightarrow P(m)$, with $m = m_1 + \cdots + m_l$, and unit $\eta : \mathbf{1} \longrightarrow P(1)$, satisfying constraints of equivariance, associativity, and unit (see \cite{MSS}). Let us denote by $\Op^{\mathcal{C}}$ the category of operads in $\C$. If $\C$ is a \textit{closed} symmetric monoidal category with internal hom $[- \ , -]$, for each object $X \in \C$ we have its \textit{operad of endomorphisms}, $\Endop_X (l) = [X^{\otimes l}, X]\ , l\geq 0$.

\subsubsection{} So, first of all we have:

\begin{proposition}\label{OperadsSonDeDescenso}
For every monoidal descent category $(\D,\mrm{E})$  its category of operads $\Op^\D$ has a natural descent structure, defined arity-wise.
\end{proposition}

\begin{proof} Define equivalences arity-wise:

$$
\Op^\mathrm{E} = \{ f: P \longrightarrow Q \ \mid \ f(l) : P(l) \longrightarrow Q(l) \in \mathrm{E}\ \mbox{for all} \ l \geq 1\}\, .
$$

Every monoidal functor $F : \D \longrightarrow \D'$ induces a functor between the categories of operads arity-wise as well: $\Op^F : \Op^\C \longrightarrow \Op^{\D'}$, which we still write $F$, $(FP)(l) = F(P(l))$; and every monoidal natural transformation $\tau : F \longrightarrow G$ induces a natural transformation $\Op^\tau : \Op^F \longrightarrow \Op^G$, which we still write $\tau$, $\tau_P(l) = \tau_{P(l)} $. Thus, since $\Simpl \Op^\D = \Op^{\Simpl \D}$, we have $\Op^\simple, \Op^\mu , \Op^\lambda$ defined for $\Op^\D$. Which endow $\Op^\D$ with a descent structure. Indeed, $\sMod^{\mathcal{D}}$ is trivially a descent category with the arity-wise structure. Finally, since all the ingredients $\Op^\simple,\Op^\mathrm{E} ,\Op^\mu, \Op^\lambda$ are defined arity-wise, the forgetful functor $\psi : \Op^\D\longrightarrow \sMod^{\mathcal{D}}$ trivially satisfies the hypotheses of the transfer lemma \cite{RR}, (2.3).
\end{proof}

\begin{example} The categories of topological operads $\Op^{\Top}$, cosimplicial operads $\Op^{\Dl \A}$, dg operads $\Op^{\Cochainsk}$ and filtered dg operads $\Op^{\mbf{F}\Cochainsk}$ are descent categories.
\end{example}

\subsection{Algebras over variable operads}

On the one hand, depending on your interests, the category of operad algebras over {\it all\/} operads might be a natural place to derive functors defined on operad algebras; for instance, if you need to take into account the failure of exactitude over the coefficient operad too. On the other hand, the case of algebras over variable operads is slightly easier than that of algebras over a fixed operad. So we start endowing the former with a descent structure.

Let $P$ be an operad in a symmetric monoidal category $\C$. An {\it operad algebra\/} is an object $A \in \C$ together with a family of $\Sigma_l$-equivariant morphisms in $\C$, $ \widehat{\alpha}_A(l): P(l)\otimes A^{\otimes l} \longrightarrow A$, subject to natural associativity and unit constraints. Equivalently, if $\C$ is a closed symmetric monoidal category, it is a morphism of operads $\alpha : P \longrightarrow \Endop_A$. We will write $\Alg_P^\C$ the category of algebras over the \emph{fixed} operad $P$.

Every morphism of operads in $\C$, $f: P \longrightarrow Q$ induces a {\it reciprocal image functor\/} between the categories of $Q$-algebras and $P$-algebras $f^* : \Alg^\C_Q \longrightarrow \Alg^\C_P$ defined on objects by the compositions $\widehat{\alpha}_B (l) = \widehat{\beta} (l) \circ (f(l)\otimes \id)$.

Loosely speaking, the category of algebras over {\it all\/} operads is the \lq union' of all $\left\{ \Alg^\C_P \right\}_{P\in \Op^\D}$. Specifically, let $\Alg^\C$ denote the category whose objects are pairs $ (P,A) $, where
$P$ is an operad in $\C$ and $A$ is a $P$-algebra. Morphisms are also couples $ (f, \varphi) : (P,A)\longrightarrow (Q,B)$, in which $f: P \longrightarrow Q$ is a morphism of operads and $\varphi : A \longrightarrow f^*(B)$ is a morphism of $P$-algebras. Composition of morphisms is done in the obvious way. $\Alg^\C$ is a {\it fibered\/} category (\cite{SGA1}) over the category of operads $\Op^\C$.

\begin{proposition}\label{AlgebrasSonDeDescenso} For every monoidal descent category $(\D,\mrm{E})$  its category of algebras over all operads $\Alg^\D$ has a natural descent structure which agrees with the one on $\D$ by forgetting the algebra structure.
\end{proposition}

\begin{proof} Define equivalences component-wise:

$$
\Alg^\mathrm{E} = \{ (f,\varphi ): (P,A) \longrightarrow (Q,B) \ \mid \ f \in \Op^\mathrm{E} \ \mbox{and}\ \varphi \in \mathrm{E} \}\ .
$$

Again, every monoidal functor $F : \D \longrightarrow \D'$ induces a functor between the categories of operad algebras $\Alg^F : \Alg^\D \longrightarrow \Alg^{\D'}$, which we still write $F$, defined on objects by $F(P,A) = (FP,FA)$, and every monoidal natural transformation $\tau : F \longrightarrow G$ induces a natural transformation $\Alg^\tau : \Alg^F \longrightarrow \Alg^G$, which we still write $\tau$, defined by $\tau_{(P,A)} = (\tau_P, \tau_A) $. So, because $\Simpl \Alg^\D = \Alg^{\Simpl \D}$, we have $\Alg^\simple, \Alg^\mu , \Alg^\lambda$ defined for $\Alg^\D$.

These data endow $\Alg^\D$ with a descent structure. Indeed, $\Op^\D\times \D$ is a descent category with the obvious product structure. Since all data are already defined via the forgetful functor $\psi : \Alg^\D \longrightarrow \Op^\D \times \D$, $\psi$ trivially satisfies the hypotheses of the transfer lemma \cite{RR} (2.3).
\end{proof}

\begin{examples}  $\Alg^{\Top}$,  $\Alg^{\Dl \A}$, $\Alg^{\Cochainsk}$ and $\Alg^{\mbf{F}\Cochainsk}$ are descent categories.
\end{examples}

\subsection{Algebras over a fixed operad}

For algebras over a fixed operad, $\Alg_P^\D$, we need the following

\begin{remark}\label{obviousremark} Let $P\in \Op^\D$ be a fixed operad. Every monoidal functor $F: \D \longrightarrow \D'$ induces a functor between the categories of operad algebras $\Alg^F_P : \Alg^\D_P \longrightarrow \Alg^{\D'}_{FP}$, which we still write $F$. Of course it is still true that every monoidal natural transformation $\tau : F \Longrightarrow G$ induces a natural transformation between the induced functors on algebras, but since both have different ranges, perhaps it is worth to explain what we mean by this. Note that $\tau_P : FP \longrightarrow GP$ is a morphism of operads. So we have a reciprocal image functor $\tau_P^* : \Alg_{GP}^{\D'} \longrightarrow\Alg_{FP}^{\D'}$, and again, the monoidal natural transformation $\tau $ induces a natural transformation $\Alg_P^\tau:  F \Longrightarrow \tau_P^* \circ G$, which we still write simply $\tau$, because, forgetting the operad algebra structures, it is the same original $\tau$.
$$
\xymatrix{
{\Alg^\D_P} \ar[r]^F \ar[d]_-G & {\Alg^{\D'}_{FP}} \\
{\Alg^{\D'}_{GP}} \ar[ur]_{\tau_P^*}
}
$$
\end{remark}

That said, the proof for algebras over a fixed operad goes along the same lines as the ones for operads and algebras over variable operads (\ref{OperadsSonDeDescenso}) and (\ref{AlgebrasSonDeDescenso}).

\begin{proposition}\label{PAlgebrasSonDeDescenso} For every operad $P \in \Op^\D$, $\Alg_P^\D$ has a natural descent structure which agrees with the descent one on $\D$ by forgetting the algebra structure.
\end{proposition}

\begin{proof} Take as equivalences

$$
 \Alg^\mathrm{E}_P = \{ \varphi : A \longrightarrow B \ \mid \mbox{the underlying morphism of $\varphi$ in $\D$ belongs to } \mathrm{E} \}\ .
$$

A cosimplicial $P$-algebra in $\D$ is the same as an ordinary $cP$-algebra in $\Dl \D$; that is, an algebra over the constant cosimplicial operad $cP$. So our simple functor induces a functor between the categories of algebras $\Alg^\simple_{cP} : \Delta\Alg^\D_P = \Alg^{\Delta\D}_{cP}\longrightarrow \Alg^\D_{\simple (cP)} $, which we compose with the reciprocal image functor induced by the morphism of operads $\lambda_P : P \longrightarrow \simple (cP)$ to obtain the simple functor for $P$-algebras $\simple_P  = \lambda_P^* \circ \Alg^\simple_{cP} : \Delta\Alg^\D_P \longrightarrow \Alg_P^\D$.

Our previous remarks \ref{obviousremark} provide us with natural transformations $\mu_P$ and $\lambda_P$ as well. Again, since all data
$ \Alg^\mathrm{E}_P, \simple_P, \mu_P$ and $\lambda_P$ are defined through the forgetful functor $\psi : \Alg_P^\D \longrightarrow \D$,
they satisfy trivially the hypotheses of the transfer lemma \cite{RR} (2.3).
\end{proof}

\begin{examples} For $P \in \Op^\Top, \Op^{\Dl \A}, \Op^{\Cochainsk}$ and $\Op^{\mbf{F}\Cochainsk}$, the respective $P$-algebras $\Alg^{\Top}_P, \Alg^{\Dl \A}_P, \Alg^{\Cochainsk}_P$ and $\Alg^{\mbf{F}\Cochainsk}_P$ are descent categories.
\end{examples}

\section{Products in sheaf cohomology}

\subsection{Sheaves of operads and operad algebras}

\subsubsection{} We are now ready to endow sheaves of operads and operad algebras with Cartan-Eilenberg structures and obtain, in particular, the derived functors of direct images and sheaf cohomology therein.

\begin{hypotheses}\label{hipotesis11} For that, we need the Grothendieck site $\X$ and the descent category of coefficients $\D$ to verify the hypotheses (\ref{hipotesis1}) of our main theorem (\ref{maintheoremRR}) and to add a new one (G3) that makes filtered colimits a monoidal functor (see \cite[1.2.2]{F}):

\begin{enumerate}
\item[(G0)] $\X$ is a Grothendieck site with a set $X$ of enough points.
\item[(G1)] $\D$ is closed under filtered colimits and arbitrary products.
\item[(G2)] Filtered colimits and arbitrary products in $\D$ are $\mathrm{E}$-exact.
\item[(G3)] For every object $A \in \D$, $A\otimes- :\D\longrightarrow \D$ preserves filtered colimits.
\end{enumerate}

\end{hypotheses}

With (G3) in hand we can show that also $\Op^{\D}$, $\Alg^\D$ and $\Alg^\D_P$ satisfy hypotheses (G1) and (G2). Namely,

\begin{lemma}\label{lemaHipotesisAlg} Let $\D$ be a descent category satisfying hypotheses (\ref{hipotesis11}). Then

\begin{enumerate}
\item Products and filtered colimits exist in $\Op^{\D}$ and are defined arity-wise.
\item Products and filtered colimits exist in $\Alg^\D$ and are preserved by the forgetful functor $\Alg^\D\longrightarrow \Op^{\D}\times\D$.
\item Products and filtered colimits exist in  $\Alg^\D_P$ and are preserved by the forgetful functor $U : \Alg^\D\longrightarrow \D$, for every operad $P$.
\end{enumerate}

In particular, $\Op^{\D}$, $\Alg^\D$ and $\Alg^\D_P$ satisfy hypotheses (G1) and (G2).
\end{lemma}

\begin{proof} It is a standard fact that limits of operads (resp. algebras) are created arity-wise (resp. by forgetting the algebra structure), and that the same is true for filtered colimits if (G3) holds (see, for instance, \cite[3.1.6, 3.3.1]{F}). Then \textit{(1)}, \textit{(2)} and \textit{(3)} hold. Finally, since filtered colimits, products and weak equivalences in $\Op^{\D}$ (resp. $\Alg^\D$, $\Alg^\D_P$) are defined arity-wise (resp. by forgetting the algebra structure), then (G2) also holds for $\Op^{\D}$ (resp. $\Alg^\D$, $\Alg^\D_P$).
\end{proof}

As a consequence, we have all the necessary ingredients to extend our main theorem \ref{maintheoremRR} to the categories of operads and operad algebras $\E = \Op^{\D}, \Alg^\D$ and $\Alg^\D_P$:

\begin{enumerate}
\item We can define the class of strong $\Scal$ and weak $\W$ equivalences as follows: a morphism $\varphi \in \Sheaf (\X, \D)$ is in $\Scal$ if and only if $\varphi(l)(U) \in \mathrm{E}$ for all objects $U\in \X$ and arities $l$, in the case of operads; and $\varphi (U) \in \mathrm{E}$ in the case of algebras. As for $\W$ we ask $\varphi(l)_x \in \mathrm{E}$ for all points $x\in X$ and arities $l$, in the case of operads; and $\varphi_x\in \mathrm{E}$ in the case of algebras. Obviously $\Scal \subset \W$, since it is true arity-wise for operads, or by forgetting the algebra structure for algebras.
\item We can define Godement cosimplicial resolutions $G^\bullet : \Sh{\X}{\E}\longrightarrow \Dl \Sh{\X}{\E}$ and hypercohomology sheaves: arity-wise for operads $\left(\mbb{H}_{\X} (\mc{P})\right)(l) = \mbb{H}_{\X} (\mc{P}(l)) $, or just as plain sheaf $\mbb{H}_{\X} (\A)$ for algebras, and the algebra structure open-wise for both.
\end{enumerate}

\begin{remark} Perhaps this second point deserves further elaboration: how multiplicative structures are transferred from sheaves of operads and algebras to their Godement resolutions and hypercohomology sheaves, even though we cannot say right now that the Godement resolution is a monoidal functor.

Recall that the Godement resolution $G^\bullet : \Sheaf (\X, \D) \longrightarrow \Delta \Sheaf (\X, \D)$ is defined as $G^p(\F) = T^{p+1}(\F)$, where $T: \Sheaf (\X, \D) \longrightarrow \Sheaf (\X, \D)$ is the triple associated to the pair of adjoint functors $p^*\F =  (\F_x)_{x\in X} $ and $ p_*D = \prod_{x\in X} x_* (D_x)$. Here, for each point $x\in X$, the fibre $x^*$ and skyscraper $x_*$ are the couple of adjoint functors which can be computed as:

$$x^*\F = \F_x = \dirlim_{(U,u)} \F (U) \ \qquad \text{and} \qquad \ (x_* D ) (U) = \prod_{u \in x^*(\y U)} D_u \ ,
$$

(see \cite{RR} for details). Let's illustrate the transfer of multiplicative structures with the case of monoids $\Mon{\D}$ of a symmetric monoidal category $\D$ satisfying (G1) and (G3): let $\A \in \Sheaf (\X, \Mon{\D})$ be a sheaf of monoids. In particular, it is a functor $\A : \X^\op \longrightarrow \Mon{\D}$. Hence, for each object $U\in \X$, we have a product $m(U) : \A(U)\otimes \A (U) \longrightarrow \A(U)$. Now, because of (G3), those $m(U)$ induce products on the fibers $\A_x \otimes \A_x \stackrel{\kappa_x}{\longrightarrow} (\A\otimes\A)_x \stackrel{x^*(m)}{\longrightarrow} \A_x$, where $\kappa_x$ is the Künneth morphism of $\dirlim$. Also, since Cartesian products are always monoidal, for every monoid $m: D \otimes D \longrightarrow D$ in $\D$, we have naturally induced products on the skyscraper sheaves $(x_*D)(U) \otimes (x_*D) (U) \longrightarrow (x_*D) (U)$. Clearly these products pass to functors $p^*$ and $p_*$ recalled above and therefore to $T$ and $G^\bullet$ \emph{object-wise}. Finally, if $\simple$ is a monoidal functor, since it's defined object-wise for sheaves $\simple (\F)(U) = \simple (\F(U))$, we will also have an object-wise induced product on the hypercohomology sheaf $\mbb{H}_{\X}$.

Much in the same way, for sheaves of operads $\Pcal$ and operad algebras $\A$ we obtain induced object-wise structure morphisms $ \Pcal(l)(U)\otimes \Pcal(m_1) (U)\otimes \dots \otimes \Pcal(m_l)(U) \longrightarrow \Pcal(m)(U)$ and $ \Pcal(l)(U)\otimes \A(U)^{\otimes l} \longrightarrow \A(U)$.
\end{remark}

\subsubsection{} With no extra work, we can prove now a multiplicative version of our main theorem (\ref{maintheoremRR}), transferring the compatibility between $(\X, \D)$ to $(\X, \E)$:

\begin{theorem}\label{hacesoperadyalgopsonCE} For every compatible pair $(\X, \D)$ of a Grothendieck site and a monoidal descent category satisfying hypotheses $(\ref{hipotesis1})$, $\Op^{\D}$, ${\Alg^\D}$ and ${\Alg^\D_P}$ are also compatible with $\X$.
\end{theorem}

\begin{proof} We have already seen that $\Op^\D, \Alg^\D$ and $\Alg^\D_P$ are descent categories (see propositions (\ref{OperadsSonDeDescenso}), (\ref{AlgebrasSonDeDescenso}) and (\ref{PAlgebrasSonDeDescenso})). We have just shown also (\ref{lemaHipotesisAlg}) that they verify hypotheses (\ref{hipotesis1}). So, according to our main theorem (\ref{maintheoremRR}), it suffices to prove that the simple functor weakly commutes with stalks.

But this is clear for operads, since for every cosimplicial sheaf of operads $\Pcal^\bullet$ the canonical map $\simple (\Pcal^\bullet)_x \longrightarrow \simple(\Pcal^\bullet_x) $ is just $\simple (\Pcal(l)^\bullet)_x \longrightarrow \simple(\Pcal(l)^\bullet_x) $ in arity $l$, and this one belongs to $\mrm{E}$ because, by assumption, $\D$ is compatible with $\X$. And it's even clearer for algebras because this canonical map is just the same map as objects of $\D$.
\end{proof}

\begin{examples}\label{maintheorem2} The following categories are compatible with any site $\X$:
\begin{enumerate}
\item $\Op^{\Dl \A}$, $\Alg^{\Dl \A}$ and $\Alg^{\Dl \A}_P$ for any $P\in \Op^{\Dl \A} $. Here $\A$ is a symmetric monoidal abelian category satisfying $($AB4$)^{\ast}$ and $($AB5$)$.
\item $\Op^{\Cochainsk}$, $\Alg^{\Cochainsk}$ and $\Alg^{\Cochainsk}_P$ for any $P\in \Op^{\Cochainsk} $.
\item Their filtered versions $\Op^{\mbf{F}\Cochainsk}$, $\Alg^{\mbf{F}\Cochainsk}$ and $\Alg^{\mbf{F}\Cochainsk}_P$ for any $P\in \Op^{\mbf{F}\Cochainsk}$.
\end{enumerate}
\end{examples}

\begin{example} $\Op^{\Top}$, $\Alg^{\Top}$ and $\Alg^{\Top}_P$ are compatible with any site $\X$ of finite cohomological dimension for any $P\in \Op^{\Top}$. Examples of sites of finite cohomological dimension include the small Zariski site of a noetherian topological space of finite Krull dimension, the big Zariski site of a noetherian scheme $X$ of finite Krull dimension, or the small site of a topological manifold of finite dimension. The fact that $\Top$ is compatible with such $\X$ may be proved combining the following three facts: 1) pointed simplicial sets are compatible with $\X$ (\cite{RR}), 2) weak equivalences in $\Top$ are created by the singular chain functor $S : \Top \longrightarrow \Sset$ and 3) $S$ preserves $\holim$, products and filtered colimits.
\end{example}

So, for derived functors we have as in \cite{RR} corollary (4.4.1),

\begin{corollary}\label{existenciaderivadoimagenesdirectas} Let $f: \X \longrightarrow \Y$ be a continuous functor of Grothendieck sites and $(\D, \mathrm{E})$ a descent category compatible with the site $\X$ satisfying hypotheses (\ref{hipotesis11}). Then, for $\E = \Op^{\D}, \Alg^\D$, and $\Alg^\D_P$, the direct image functor $f_* : \Sh{\X}{\E} \longrightarrow \Sh{\Y}{\E}$ admits a right derived functor $ \rdf f_* : \Sh{\X}{\E}[\W^{-1}] \longrightarrow \Sh{\Y}{\E}[\W^{-1}]$ given by $\rdf f_* (\F) = f_* \mbb{H}_{\X} (\F)$.

Also, for any object $U\in \X$, the section functor $\Gamma (U, -) : \Sh{\X}{\E}
\longrightarrow \E$ admits a right derived functor $\rdf \Gamma (U,-)  : \Sh{\X}{\E}[\W^{-1}] \longrightarrow \E [\mathrm{E}^{-1}]$ given by $\rdf \Gamma (U,\F) = \Gamma (U,\mbb{H}_{\X} (\F))$. In particular, if $\X$ has a final object $X$, sheaf cohomology can be computed as $\rdf \Gamma (X, \F ) =  \Gamma (X, \mathbb{H}_{\X}(\F))$.
\end{corollary}

\begin{examples}\label{NH} (\cite{N}, \cite{Hin}, \cite{Be}) For instance, we can take $\D = \Cochainsk$. So sheaves of commutative or Lie algebras,
$\Sh{\X}{\Alg^{\Cochainsk}_{\Com}}$ and $\Sh{\X}{\Alg^{\Cochainsk}_{\Lie}}$ form CE-categories and for any sheaf of commutative $\A$ or Lie $\Lcal$ dg algebras, we have in particular sheaf cohomologies $\rdf \Gamma (X, \A) $ and
$\rdf \Gamma (X, \Lcal)$, which are commutative and Lie dg algebras, respectively. The first agrees with the \emph{Thom-Whitney} derived functor $\rdf \Gamma_{\TW} (X, \A)$ introduced by
V. Navarro in \cite{N} (where the filtered case is also considered).

A particular example of $\rdf \Gamma (X, \Lcal)$ is the dg Lie algebra which governs formal deformations of sheaves of operad dg algebras $\A$ on a site, $T_\A$, called the \textit{global tangent Lie algebra} in \cite{Hin}. It turns out that $T_\A = \rdf \Gamma (\X, \T_\A)$, where $\T_\A$ is a sheaf of dg algebras of derivations of a model of $\A$.

In \cite{Be}, sheaf cohomology $\rdf \Gamma (X, \Lcal)$ is defined using pro-hypercoverings and  the main result in algebraic deformation quantization \cite{Ye}, theorem (0.2), is written as a \quis\ between such Lie dg algebras \cite{Be}, theorem (1.1).
\end{examples}

\subsubsection{} In the next sections we develop three variants of the results we have obtained, but so far not covered. Namely, we examine

\begin{enumerate}
\item Sheaves of algebras over a \emph{fixed sheaf} of operads.
\item Symmetric spectra.
\item The case where an associated sheaf functor is available.
\end{enumerate}

We will first elaborate briefly why these three cases deserve particular treatment.

First of all, our main result (\ref{maintheorem2}) states that, under some mild hypotheses, categories of sheaves $\Sheaf (\X, \E)$ for $\E = \Op^{\D}, \Alg^\D$ and $\Alg^\D_P$ have a CE-structure; in particular, for algebras the categories of sheaves of operad algebras over \emph{variable sheaves} of operads $\Sheaf (\X, \Alg^\D)$ and over a \emph{fixed}, \emph{constant} operad $P\in \Op^\D$, $\Sheaf (\X, \Alg^\D_P)$. But the category of operad algebra sheaves over a \emph{fixed sheaf} of operads $\Pcal \in \Sheaf (\X, \Op^\D)$ is none of these, so \cite{KM} would not be covered by our construction. Therefore we need to develop this case independently.

Secondly, sheaves of symmetric spectra (and hence operads and algebras over them) are not covered  by theorem (\ref{maintheorem2}) either, because cartesian products in $\Sp^{\Sigma}$ are not $\mathrm{E}$-exact \cite[3.8]{S} as our hypothesis (G2) requires. To avoid this inconvenience, we might restrict ourselves to the subcategory ${\Omega}\Sp^{\Sigma}$ of symmetric $\Omega$-spectra, in which cartesian products do preserve weak equivalences. But then we would not have a monoidal category, because the smash product of two $\Omega$-spectra is not necessarily an $\Omega$-spectrum (\cite[1.4]{HSS}). However, this kind of difficulty can be easily overcome by taking the best of both worlds: the exactitude of products in ${\Omega}\Sp^{\Sigma}$ and the monoidal structure in $\Sp^{\Sigma}$. This works because, in order to talk about operads in ${\Omega}\Sp^{\Sigma}$, for instance, we actually \emph{don't} need the smash product of two omega spectra to be an omega spectrum, but only its existence as a plain spectrun. In such a case, we will say that ${\Omega}\Sp^{\Sigma}$ is a monoidal descent \emph{subcategory}.

Thirdly, in the presence of an associated sheaf functor, everything can be more easily stated: for instance, we will be able to say simply that the Godement resolution is a monoidal functor and forget about working with products object-wise, as we have been doing so far.

\subsection{Sheaves of algebras over a fixed sheaf of operads}\label{fijo} We now turn to the study of sheaves of algebras over a \textit{fixed} sheaf of operads.

Let $\Pcal \in \Sheaf (\X, \Op^\D)$ be one such a sheaf. Notice that we cannot define a (sheaf) of $\Pcal$-algebras outright as a sheaf $\A \in \Sheaf (\X, \D)$ together with some structure morphisms $\widehat{\alpha}_\A(l) : \Pcal(l) \otimes \A^{\otimes l} \longrightarrow \A$ because, in the absence of an associate sheaf functor, we don't have tensor products of sheaves. Nevertheless we can go on again object-wise: it always make sense to talk about \emph{presheaves} of algebras over $\Pcal$. By definition a presheaf of $\Pcal$-algebras is a presheaf $\A \in \PrSheaf (\X, \D)$ together with object-wise structure morphisms $ \widehat{\alpha}_{\A}(l)(U):
\mc{P}(l)(U)\otimes \A(U)^{\otimes l} \longrightarrow \A(U)$ natural in $U$ making $\A(U)$ a $\mc{P}(U)$-algebra for every object $U\in\X$. Let us write $\Alg^{\PrSh{\X}{\D}}_{\mc{P}}$ for the category of presheaves of $\Pcal$-algebras.

\begin{definition}\label{defiAlgSinMonoidal}
Given a sheaf of operads $\mc{P}\in\Sh{\X}{\Op^\D}$, a \textit{sheaf of $\mc{P}$-algebras} is a sheaf $\A\in\Sh{\X}{\D}$ such that, as a presheaf, $\A\in\Alg^{\PrSh{\X}{\D}}_{\mc{P}}$. Let us write $\Alg_{\mc{P}}^{\Sh{\X}{\D}}$ for the full subcategory of $\Alg^{\PrSh{\X}{\D}}_{\mc{P}}$ of sheaves $\mc{P}$-algebras.
\end{definition}

\begin{remark} This notation $\Alg_{\mc{P}}^{\Sh{\X}{\D}}$, with $\Pcal \in \Sheaf (\X, \Op^\D)$, is somewhat misleading since it is reminiscent of the previously used  $\Alg^\D_P$, where $\D$ is a (closed) symmetric monoidal category and $P \in \Op^\D$. In the absence of an associated  sheaf functor, we cannot say that $\Sheaf (\X, \D)$ is such a category, but  remark (\ref{hacesmonoidal}) should convince the reader that, in the end, it will turn out to be a coherent notation.
\end{remark}

\begin{example}\label{KM} (\cite{KM}) Let $X$ be a smooth complex algebraic variety of finite dimension and $ \Pcal_X (n)$ denote the sheaf of $n$-linear multidifferential operators. For instance, $\Pcal_X (1)$ is the usual sheaf of linear differential operators on $X$. $\Pcal_X = \left(\Pcal_X(n)\right)_{n\geq 0}$ is a sheaf of operads on $X$, with the composition given by superposition of multilinear differential operators. So we have $\Pcal_X \in \Sheaf (X, \Op^\D)$, with $\D =  \mathbf{C}^{\geq 0}(\mathbb{C})$. Then, a system of partial differential equations on $X$ (polynomial in the derivatives) is a sheaf $\A$ of $\Pcal_X$-algebras, $\A \in \Alg_{\Pcal}^{\Sh{\X}{\D}}$, and a solution to such a system is a morphism of $\Pcal_X$-algebras $\A \longrightarrow \Ocal_X$.
\end{example}

For $\Alg_{\mc{P}}^{\Sh{\X}{\D}}$ we have the same ingredients as for the categories of our main theorem (\ref{maintheorem2}), namely, the Godement resolution and classes of strong and weak equivalences. As for the first, the Godement resolution and the hypercohomology sheaf of $\Sh{\X}{\D}$ induce resolution and hypercohomolgy for $\Alg^{\Sh{\X}{\D}}_{\mc{P}}$ in a natural way.

\begin{lemma}\label{lemaCEAlgP}
Let $\D$ be a monoidal descent category satisfying hypotheses $(\ref{hipotesis1})$. Then, for a sheaf of $\mc{P}$-algebras $\A$, we have:
\begin{itemize}
 \item[(1)] the hypercohomology sheaf $\mbb{H}_{\X} \A$ is a sheaf of $\mc{P}$-algebras, and
 \item[(2)] the morphism $\rho_{\A}: \A \longrightarrow \mbb{H}_{\X} \A$ is a morphism of $\mc{P}$-algebras.
\end{itemize}
\end{lemma}

\begin{proof} Recall from remark (3.1.2) that $G^{\bullet}$ is built up  using only cartesian products and filtered colimits. The former products are always monoidal and filtered colimits are too because of our hypothesis (G3). Hence, for any sheaf $\A$ of $\mc{P}$-algebras,
the cosimplicial sheaf $G^{\bullet}\A$ is a $G^{\bullet}\mc{P}$-algebra. Also, $\mbf{s}:\Dl\D\longrightarrow\D$ is monoidal,
and $\mbf{s}:\Dl\Sh{\X}{\D}\longrightarrow\Sh{\X}{\D}$ is defined object-wise, so
$\mbb{H}_{\X} \A = \mbf{s} G^{\bullet}\A $ is a sheaf of $\mbb{H}_{\X}\mc{P}$-algebras, and, thanks to $\rho_\mc{P} : \mc{P} \longrightarrow \mbb{H}_{\X}\mc{P}$, we get that
$\mbb{H}_{\X} \A\in\Alg_{\mc{P}}^{\Sh{\X}{\D}}$ as required.

Finally, recall that the morphism of sheaves $\rho_\A : \A \longrightarrow \mbb{H}_{\X}\A$ is the composition of the natural transformation $\lambda$ (which is object-wise monoidal because $\D$ is a monoidal descent category) with the simple of the canonical augmentation $\id \longrightarrow G^\bullet$, which is object-wise monoidal by construction. Hence, $\rho_\A$ is a morphism of sheaves of $\mc{P}$-algebras.
\end{proof}

As for $\Scal$ and $\W$: call a morphism of sheaves of $\mc{P}$-algebras a \emph{global equivalence} (resp. \emph{local equivalence}) if it is a global one (resp, a local one) as a morphism of plain sheaves in $\Sheaf (\X, \D)$. Obviously, $\Scal \subset \W$, and we get the following version of our main result in this setting:

\begin{theorem}\label{CEAlgP} For every compatible pair $(\X, \D)$  satisfying hypotheses $(\ref{hipotesis1})$ and  any sheaf of operads $\mc{P}$, $(\Alg^{\Sh{\X}{\D}}_{\mc{P}},\Scal,\W)$ is a Cartan-Eilenberg category and $\rho_{\A}: \A \longrightarrow \mbb{H}_{\X} \A$ is a CE-fibrant model for any sheaf of $\mc{P}$-algebras $\A$.
\end{theorem}

\begin{proof} Since $\D$ is compatible with the site $\X$,
$(\Sh{\X}{\D},\Scal,\W)$ is a Cartan-Eilenberg category and
$\rho_{\F}: \F \longrightarrow \mbb{H}_{\X} \F$ is a CE-fibrant model for any sheaf $\F$. Thus $(\mbb{H}_{\X} ,\rho_{\F})$ is a resolvent functor for $(\Sh{\X}{\D},\Scal,\W)$ and this in particular implies that
(1) $\mbb{H}_{\X}^{-1}(\mc{S})=\W$, (2) $\mbb{H}_{\X}(\Scal)\subset\Scal$ and (3) for any sheaf $\F$ it holds that $\rho_{\mbb{H}_{\X} \F}$ and $\mbb{H}_{\X} (\rho_{\F})$ belong to $\mc{S}$. By lemma (\ref{lemaCEAlgP}), $\mbb{H}_{\X}$ and $\rho: \mrm{id} \longrightarrow \mbb{H}_{\X} $ restrict to $\Alg^{\Sh{\X}{\D}}_{\mc{P}}$.
Consequently, the above properties (1)-(3) also hold for sheaves of $\mc{P}$-algebras, and in that case, the result follows directly from
\cite[Theorem 2.5.4]{GNPR}.
\end{proof}

Hence we obtain properties (\ref{PropertiesCompat}) for $\Alg^{\Sh{\X}{\D}}_{\mc{P}}$ too. Also, it follows directly from the definitions that the direct image functor $f_*$ associated to a morphism of sites maps a sheaf of $\mc{P}$-algebras to a sheaf of $f_*\mc{P}$-algebras, and

\begin{corollary} Let $f: \X \longrightarrow \Y$ be a continuous functor of Grothendieck sites and $(\D, \mathrm{E})$ a descent category compatible with the site $\X$ and satisfying hypotheses (\ref{hipotesis11}). Then the direct image functor
$f_* : \Alg_{\mc{P}}^{\Sh{\X}{\D}} \longrightarrow \Alg_{f_*\mc{P}}^{\Sh{\Y}{\D}}$
admits a right derived functor given by $\rdf  f_* (\A) = f_* \mbb{H}_{\X} (\A)$. In particular, if $\X$ has a final object $X$, sheaf cohomology can be computed as $\rdf\Gamma(X,\A) = \Gamma(X,\mbb{H}_{\X}(\A))$ and $\rdf \Gamma (X, \A)$ is a $\rdf \Gamma (X, \Pcal)$-algebra.
\end{corollary}

\subsection{Symmetric spectra and further examples}\label{espectros} Here we generalize previous results to the case of sheaves of operads and operad algebras in a category $\D$ which is not a monoidal category itself, but is a full subcategory of a monoidal category.

\begin{definition} Let $\C$ be a symmetric monoidal category and $\D \subset \C$ a full subcategory. The category of \textit{operads in} $\D$, $\Op^{\D}$, is the full subcategory of $\Op^{\C}$ whose objects are operads $P$ such that $P(n)$ belongs to $\D$ for all $n$. Analogously, $\Alg^\D$ and $\Alg^\D_P$ for a fixed operad $P$ in $\C$ are defined, respectively, as the full subcategories of $\Alg^\C$ and $\Alg^\C_P$, which objects are the algebras that belong to $\D$ after forgetting the algebra structure.
\end{definition}

In this way we may consider the categories $\Op^{{\Omega}\Sp^{\Sigma}}$,  $\Alg^{{\Omega}\Sp^{\Sigma}}$, or $\Alg^{{\Omega}\Sp^{\Sigma}}_P$, where $P$ is an  operad of symmetric $\Omega$-spectra.

\begin{definition} Let $\D$ be a full subcategory of a symmetric monoidal category $\C$.  We say that $\D$ is a \emph{monoidal descent subcategory} if $\D $ admits a descent structure $(\mrm{E},\mbf{s},\mu,\lambda)$ such that the triple $(\mbf{s},\mu,\lambda)$ extends to the ambient monoidal category $\C$ and it is compatible there with the symmetric monoidal structure on $\C$ as in (\ref{MonoidalDescentDefi}).
\end{definition}

\begin{remark} The point here is that $\C$ \emph{needs not} satisfy our hypotheses (\ref{hipotesis1}).
\end{remark}

\subsubsection{} Proposition (\ref{topologicalSpaces}) says that every simplicial symmetric monoidal model category in which all objects are fibrant is a monoidal descent category. Now we can go further:

\begin{proposition}
For any simplicial symmetric monoidal model category $\M$, its category of fibrant objects $\M_f$ is a monoidal descent subcategory.
\end{proposition}

\begin{proof} This works like the proof of Proposition (\ref{topologicalSpaces}), just observing that the triple $(\holim,\mu,\lambda)$ which gives a descent structure on $\M_f$ is defined on all $\M$ in the same way. That is, $\holim X=\int_n (X^n)^{\mrm{N}(\Dl\downarrow n)}$, and analogously for $\mu$ and $\lambda$.
\end{proof}

\begin{examples} We will focus on two particular instances of $\M$:
\begin{itemize}
\item Let $\M = \Sset$ be the category of pointed simplicial sets, with weak homotopy equivalences as weak equivalences and the smash product as tensor product. Then, the category  $\Sset_f$ of pointed Kan complexes is a monoidal descent subcategory.
\item Let $\M = \Sp^{\Sigma}$ be the category of  symmetric spectra, with stable equivalences as weak equivalences and the smash product as tensor product (see \cite{HSS}). Its fibrant objects $(\Sp^{\Sigma})_f ={\Omega}\Sp^{\Sigma}$,  are precisely the $\Omega$-spectra: those symmetric spectra which are level-wise Kan complexes and such that the adjoint to the structure maps $S^1 \wedge X_n \rightarrow X_{n+1}$ are weak equivalences for all $n$. Then, ${\Omega}\Sp^{\Sigma}$ is a monoidal descent subcategory.
\end{itemize}
\end{examples}

Now we can reobtain our previous results for monoidal descent subcategories $\D$. For instance,  we have natural descent structures on $\Op^\D$, $\Alg^\D$ and $\Alg^{\D}_P$.

\begin{proposition}\label{OperadsSonDeDescensoII}
For every monoidal descent subcategory $(\D,\mrm{E})$, we have
\begin{itemize}
 \item[(1)] $\Op^\D$ has a descent structure, defined arity-wise.
 \item[(2)] $\Alg^\D$  has a descent structure, com\-pa\-ti\-ble with the one on $\D$.
 \item[(3)] For any operad $P \in \Op^\C$, $\Alg^{\D}_P$ has a descent structure, com\-pa\-ti\-ble with the one on $\D$.
\end{itemize}
\end{proposition}

\begin{proof} We will see (1); the other statements are proved analogously. Since the triple $(\mbf{s},\mu,\lambda)$ is symmetric monoidal on the ambient category $\C$, it induces $(\Op^\simple, \Op^\mu , \Op^\lambda)$ defined arity-wise for $\Op^\C$. But since $(\mbf{s},\mu,\lambda)$ restricts to $\D$, $(\Op^\simple, \Op^\mu , \Op^\lambda)$ restricts to $\Op^\D$. Proof is then completed  using the transfer lemma as
in Proposition \ref{OperadsSonDeDescenso}.
\end{proof}

\begin{example} For any simplicial symmetric monoidal model category $\M$, $\Op^{\M_f}$, $\Alg^{\M_f}$ and $\Alg^{\M_f}_P$ are descent categories.
\end{example}

Again the compatibility of the base category $\D$ with a given site is transferred to operads and operad algebras in $\D$:

\begin{theorem} For every compatible pair $(\X, \D)$ of a Grothendieck site and a monoidal descent subcategory satisfying hypotheses $(\ref{hipotesis1})$, $\Op^{\D}$, ${\Alg^\D}$ and ${\Alg^\D_P}$ are also compatible with $\X$.
\end{theorem}

\begin{proof} The proof is completely analogous to the proof of theorem (\ref{hacesoperadyalgopsonCE}).
\end{proof}

\begin{example} We return  to the example of a symmetric monoidal simplicial model category $\M$. If, in addition, filtered colimits in $\M$ preserve fibrant objects, weak equivalences between them, and tensor products, then $\Op^{\M_f}$, $\Alg^{\M_f}$ and $\Alg^{\M_f}_P$ are compatible with any site $\X$ for which ${\M_f}$ was already compatible.
In particular, we obtain
\end{example}

\begin{corollary}
For any site $\X$ of finite cohomological dimension, and for $\M_f =\Sset_f$, or ${\Omega}\Sp^{\Sigma}$, $\Op^{\M_f}$, $\Alg^{\M_f}$ and $\Alg^{\M_f}_P$ are compatible with $\X$.
\end{corollary}

\begin{proof} Follows from the previous theorem and the proposition below.
\end{proof}

\begin{proposition} Any finite cohomological dimension site $\X$ is compatible with the monoidal descent subcategories $\Sset_f$ and ${\Omega}\Sp^{\Sigma}$.
\end{proposition}

\begin{proof} For $\Sset_f$, see \cite{RR}, theorem (5.7). As for ${\Omega}\Sp^{\Sigma}$, it is enough to show that for every sheaf $\F\in \Sh{\X}{{\Omega}\Sp^{\Sigma}}$, $\rho_\F : \F\longrightarrow \mbb{H}_{\X} (\F)$ is in $\mc{W}$. Since  stable equivalences between $\Omega$-spectra are the same as level-wise weak equivalences of simplicial sets, it suffices to see that $(\rho_\F)_n$ is a weak equivalence of simplicial sets for all $n\geq 0$. In $\Sp^{\Sigma}$, simplicial action, limits and colimits are all defined by extension from the same operations on pointed simplicial sets, and hence they are given arity-wise (see \cite{HSS}). Then, $\mbb{H}_{\X}(\F)_n = \mbb{H}_{\X}(\F_n)$ and $(\rho_\F)_n$ agrees with the corresponding morphism of simplicial sets $\rho_{\F_n}$. But this is a weak equivalence because by \cite[5.3]{RR} $\Sset_f$ is compatible with $\X$, so we are done.
\end{proof}

\subsection{Monoidal structures for categories of sheaves} So far we have carefully avoided  the use of the associated sheaf, but now we are going to free ourselves of this coarse limitation.

\begin{hypotheses}\label{hipotesis2} For this, we have to slightly modify our hypothesis (G1) in (\ref{hipotesis1}) and add a new hypothesis (G4):

\begin{enumerate}
\item[(G0)] $\X$ is a Grothendieck site with a set $X$ of enough points.
\item[(G1')] $\D$ is closed under filtered colimits and arbitrary limits.
\item[(G2)] Filtered colimits and arbitrary products in $\D$ are $\mathrm{E}$-exact.
\item[(G3)] For every object $A \in \D$, $A\otimes- :\D\longrightarrow \D$ preserves filtered colimits.
\item[(G4)] Filtered colimits commute with finite limits in $\D$.
\end{enumerate}

\end{hypotheses}

Hypotheses (G1') and (G4) guarantee the existence of the \emph{associated sheaf} functor $(\  )^a: \PrSheaf (\X , \D) \longrightarrow   \Sheaf (\X, \D )$ (see \cite{Ul}).

Then everything is streamlined because, as a consequence, $\Sheaf (\X, \D)$ is a (closed) monoidal category when $\D$ is, and we can say the  Godement resolution and the direct image functor are outright monoidal functors, and the detour of stating this object-wise, as we have been doing so far, becomes unnecessary.

\begin{lemma}\label{shsymclosed} For every (closed) monoidal category $\D$ satisfying hypothesis (G1') and (G4) in \emph{(\ref{hipotesis2})}, $\Sh{\X}{\D}$ is also a (closed) monoidal category.
\end{lemma}

\begin{proof} If $\F, \G \in \Sh{\X}{\D}$, their tensor product $\F \otimes \G$ is the sheaf associated to the presheaf $ U \mapsto \F (U) \otimes \G (U)$. As for the inner hom functor, consider the presheaf
$$
U \mapsto \int_{V \rightarrow U} [\F (V) , \G (V) ] \ ,
$$
where the end is taken over all maps $V \longrightarrow U \in J(U)$. Then, the inner hom functor $ \Homsheaf (\F , \G)$ is the sheaf associated to this presheaf.
\end{proof}

\begin{remark}\label{hacesmonoidal} An operad in $\Sh{\X}{\D}$ is the same as a sheaf of operads in $\D$. That is,  $\Op^{\Sh{\X}{\D}}=\Sh{\X}{\Op^\D}$. Also an operad algebra in $\Sh{\X}{\D}$ is the same as a sheaf of operad algebras in $\D$; that is, $\Alg^{\Sh{\X}{\D}}=\Sh{\X}{\Alg^\D}$. Moreover, for any  operad of sheaves $\Pcal$, the category $\Alg_\Pcal^{\Sh{\X}{\D}}$ of $\Pcal$-algebras in  $\Sheaf (\X, \D)$ agrees with the category denoted in the same way as in definition (\ref{defiAlgSinMonoidal}).
\end{remark}

Next we prove the monoidalness of the two functors appearing in the Godement resolution.

\begin{lemma} Let $\D$ be a symmetric monoidal category satisfying hypotheses (G1) and (G3). Then, the pair of adjoint functors
$$
\xymatrix{ {\Sh{\X}{\D} }   \ar@<0.5ex>[r]^-{p^*}  &  {\D^X}
\ar@<0.5ex>[l]^-{p_*}
}
\ , \qquad \qquad \D^X (p^*\F , D) = \Sh{\X}{\D} (\F , p_*D)
$$
are monoidal ones.
\end{lemma}

\begin{proof}
Let us first show that the pair of adjoint functors induced by points on the Grothendieck site
$$
\xymatrix{ {\Sh{\X}{\D}}  \ar@<0.5ex>[r]^-{x^*}  &  {\D}
\ar@<0.5ex>[l]^-{x_*}
}
\ , \qquad \qquad \D (x^*\F , D) = \Sh{\X}{\D} (\F , x_*D)
$$
are monoidal ones. Let us compute:
\begin{eqnarray*}
x^*(\F) \otimes x^*(\G) &=& \left( \dirlim_U \F (U)\right) \otimes \left( \dirlim_U \G (U) \right)
        \cong  \dirlim_U \left( \F (U) \otimes \G (U) \right)  \\
        &\longrightarrow& \dirlim_U (\F \otimes \G) (U)
        = x^* (\F \otimes \G ) \ .
\end{eqnarray*}
Here, the arrow is the one induced by the canonical map between a presheaf and its associated sheaf, $\F (U) \otimes \G (U) \longrightarrow (\F \otimes \G) (U)$, and the isomorphism follows from the assumption that filtered colimits commute with tensor products in $\D$.

As for $x_*$, we have:
$$
\xymatrix{
{\left( (x_*D) \otimes (x_*D') \right) (U)} \ar@{.>}[dd]^{\kappa_x}  & \ar[l]_-a { \left( \prod_u D_u \right)  \otimes \left( \prod_v D'_v \right) } \ar[d]^c       \\
           & {\prod_{(u,v)} \left( D_u\otimes D'_v \right) } \ar[d]^{\pi}    \\
{x_*(D\otimes D')(U)}   &   {\prod_u  \left(  D_u\otimes D'_u \right) }    \ar[l]_a
}
$$
Here $a$ is the canonical map between a presheaf and its associated sheaf, $c$ is the canonical morphism
defined in the obvious way, $\pi$ the natural projection and $\kappa$ the map induced by the universal property
of the associated sheaf.

Monoidalness of $p^* = (\F_x)_{x\in X}$ follows at once. As for $p_*D = \prod_{x\in X} x_*(D_x)$, $D = (D_x)_{x\in X}$, it is
enough to notice that products are always (lax) monoidal.
Hence, we have the following K\"{u}nneth map for $p_*$:
$$
\xymatrix{
{p_*D \otimes p_*D'} \ar@{=}[r] \ar@{.>}[d]^{\kappa_p}
            & { \left( \prod_{x\in X} x_*D \right) \otimes \left( \prod_{y\in X} y_*D' \right) } \ar[r]^-c
            & { \prod_{x,y} (x_*D) \otimes (y_*D') }  \ar[d]^{\pi}  \\
{p_* (D\otimes D')}  \ar@{=}[r]
            & {\prod_{x\in X} x_*(D\otimes D') }
            & {\prod_{x\in X} (x_*D) \otimes (x_*D') } \ar[l]_{\kappa_x}
}
$$
\end{proof}

As a consequence, we can finally state the fact that

\begin{proposition}\label{Godementisasymmetricmonoidalfunctor} $G^\bullet : \Sh{\X}{\D} \longrightarrow \Simpl\Sh{\X}{\D}$ is
a monoidal functor.
\end{proposition}

So the cosimplicial Godement resolution $G^\bullet : \Sheaf(\X, \D) \longrightarrow \Delta\Sheaf(\X,\D)$ defines functors $ \E^{G^\bullet} :\E^{\Sheaf (\X,\D)} \longrightarrow \Delta \E^{\Sheaf(\X,\D)}$ for $\E = \Op, \Alg$ and $\Alg_\Pcal$, for every $\Pcal \in \Op^{\Sheaf (\X,\D)} = \Sheaf (\X , \Op^\D)$; which we shall still denote simply by $G^\bullet$.

\begin{corollary}\label{GodementMonoidal} For any sheaf of operads $\Pcal\in \Op^\D$ and any sheaf of $\Pcal$-algebras $\A$, $G^\bullet \Pcal$ is a cosimplicial sheaf of operads and $G^\bullet \A$ a cosimplicial $G^\bullet\Pcal$-algebra.
\end{corollary}

Finally, in order to have products on the derived direct image and global section functors, we note that we already have them before deriving.

\begin{proposition}\label{ImagenDirectaMonoidal} Let $f: \X \longrightarrow \Y$ be a morphism of sites, and $\D$ a monoidal category
satisfying hypotheses \emph{(\ref{hipotesis2})}. Then, $f_* : \Sheaf (\X,\D) \longrightarrow \Sheaf (\Y,\D)$ and
$\Gamma (U, -) : \Sheaf (\X,\D) \longrightarrow \D$ are monoidal functors.
\end{proposition}

\begin{proof} Take the universal map between the presheaf and the associated sheaf:
$$
\theta_{f^{-1}(V)} : \F(f^{-1} (V)) \otimes \G (f^{-1} (V))
\longrightarrow (\F \otimes \G)(f^{-1}(V))
$$
and analogously for the global sections functor.
\end{proof}

\begin{corollary}
Let $f: \X \longrightarrow \Y$ be a morphism of sites, $\D$ a monoidal category satisfying hypotheses \emph{(\ref{hipotesis2})}, $\Pcal$ a sheaf of operads on $\X$ with values in $\D$, and $\A$ a $\Pcal$-algebra. Then:
\begin{enumerate}
\item $f_*\Pcal$ is a sheaf of operads on $\Y$ and $f_*\A$ a $f_*\Pcal$-algebra.
\item For every object $U\in \X$, $\Gamma (U, \Pcal)$ is an operad in $\D$ and $\Gamma (U,\A)$ a $\Gamma (U, \Pcal)$-algebra.
\end{enumerate}
\end{corollary}

So,  Theorem (\ref{hacesoperadyalgopsonCE}) applies and provide natural CE-structures on $\Op^{\Sh{\X}{\D}}$,  $\Alg^{\Sh{\X}{\D}}$ and $\Alg_\Pcal^{\Sh{\X}{\D}}$ allowing us to obtain $\rdf f_*$ by the usual formula.

\begin{theorem}\label{lastone} For any monoidal descent category $\D$ satisfying hypotheses \emph{(\ref{hipotesis2})} and compatible with site $\X$, the categories of operads and operad algebras $\Op^{\Sheaf (\X, \D)}$, $\Alg^{\Sheaf (\X,\D)}$ and $\Alg^{\Sheaf(\X,\D)}_\Pcal$ are also compatible with $\X$.
\end{theorem}

\begin{remark} Notice that $\rdf f_* (\Pcal, \A) = (\rdf f_* \Pcal , \rdf f_* \A)$, where the derived functors on the right are $\rdf f_* : \Sh{\X}{\Op^\D} \longrightarrow \Sh{\Y}{\Op^\D}$ and $\rdf f_*: \Sh{\X}{\D} \longrightarrow \Sh{\Y}{\D}$. This means that the derived
direct image functor can be computed by forgetting the $\Pcal$-algebra structure: it only depends on the sheaf $\A \in \Sh{\X}{\D}$. Here is another way of saying this: we have a commutative diagram
$$
\xymatrix{
{\Sh{\X}{\Alg^\D}[\W^{-1}]}  \ar[r]^{\rdf f_*} \ar[d]^{\psi}    & {\Sh{\Y}{\Alg^\D}[\W^{-1}]}    \ar[d]^{\psi}\\
{\Sh{\X}{\D}[\mathrm{\W}^{-1}]}  \ar[r]^{\rdf f_*}            & {\Sh{\Y}{\D}[\mathrm{\W}^{-1}]}
}
$$
in which $\psi$ is induced by the forgetful and exact functor $\Alg^\D \longrightarrow \D$. Nevertheless, as an object in $\Sh{\Y}{\Alg^\D}$, before localizing, $\rdf f_* \A$ has a $\rdf f_* \Pcal$-algebra structure thanks to
$$
\rdf f_* \Pcal (l) \otimes (\rdf f_* \A)^{\otimes l} \stackrel{\kappa}{\longrightarrow} \rdf f_* \left(\Pcal (l)\otimes \A^{\otimes l}\right) \stackrel{\rdf f_* \widehat{\alpha}_\A (l)}{\longrightarrow}  \rdf f_* \A  \ , \ l \geq 1 \ ,
$$
where $\rdf f_* \widehat{\alpha}_\A (l)$ are the morphisms induced by the $\Pcal$-algebra structure maps $\widehat{\alpha}_\A (l) : \Pcal (l) \otimes \A^{\otimes l} \longrightarrow \A$, $l \geq 1$, and $\kappa$ is the K\"{u}nneth morphism of the monoidal functor $\rdf f_* $.

Moreover, if we forget the operad structure,  as a $\Sigma$-module, $\rdf f_* \Pcal$ is just $(\rdf f_* \Pcal)(l) = \rdf f_* (\Pcal (l))$, $l\geq 1$, where $\rdf f_*$ on the right-hand side is again the derived functor $\rdf f_*: \Sh{\X}{\D} \longrightarrow \Sh{\Y}{\D}$, and the operad structure morphisms on $\rdf f_* \Pcal$ are induced by those of $\Pcal$ in the same way as before.

Hence, sheaf cohomology $\rdf \Gamma (X, \A)$ can be computed forgetting the operad algebra structure: it only depends on sheaf $\A \in \Sh{\X}{\D}$ (but it has the structure of a $\rdf \Gamma (X, \Pcal)$-algebra inherited from the $\Pcal$-algebra structure of $\A$). This is reminiscent of the well-known fact in classical sheaf cohomology (see, for instance, \cite{Har}, III Proposition 2.6 and Remark 2.6.1) that, if $(X , \Ocal_X)$ is a ringed space and $\F$ an $\Ocal_X$-module, $H^n(X; \F)$, despite being an $H^n (X; \Ocal_X)$-module, can be computed just with the abelian group structure of $\F$. It's also a particular case of our proposition (6.10) in \cite{RR}.
\end{remark}

\begin{remark} Theorem (\ref{lastone}) also applies to sheaves of operads and operad algebras in examples (\ref{KM}) and (\ref{NH}), since for these categories of sheaves we have an associate sheaf functor. Therefore we have, apparently, two ways of inducing products on the derived functors $\rdf f_*$ and $\rdf \Gamma$: the one obtained in section (\ref{fijo}) and the one described here. Both constructions are, of course, related: the first one is defined at the presheaf level and the latter is just the product induced on the associated sheaf by its the universal property. However, (\ref{lastone}) is not available for the cases treated in section (\ref{espectros}), because Kan complexes are not closed under limits.
\end{remark}

\subsubsection{Acknowledgements} This paper develops an idea suggested to us by Vicente Navarro. We owe him a debt of gratitude for sharing it with us. The second-named author also benefited from many fruitful conversations with Pere Pascual. We are indebted to Francisco Guill\'{e}n, Fernando Muro, Luis Narv\'{a}ez and Abd\'{o} Roig for their comments

\end{large}
\end{document}